\documentclass[10pt,british]{article}
\usepackage{amsfonts}
\usepackage{latexsym}
\usepackage{amsmath}
\usepackage{amssymb}
\usepackage{amssymb}
\usepackage{amsthm}
\usepackage{mathrsfs}
\usepackage[pagebackref,colorlinks=true]{hyperref}

\newtheorem{thrm}{Theorem}[section]

\newtheorem{lemma}[thrm]{Lemma}

\newtheorem{prop}[thrm]{Proposition}

\newtheorem{cor}[thrm]{Corollary}

\newtheorem{dfn}[thrm]{Definition}

\newtheorem{rmrk}[thrm]{Remark}

\newcommand{\newsection}{    
\setcounter{equation}{0}\section}
\def\appendix#1{\addtocounter{section}{1}\setcounter{equation}{0}
\renewcommand{\thesection}{\Alph{section}}
\section*{Appendix \thesection\protect\indent \parbox[t]{11.15cm}{#1}}
\addcontentsline{toc}{section}{Appendix \thesection\ \ \ #1}}

\newcommand{\be}{\begin{eqnarray}}
\newcommand{\ee}{\end{eqnarray}}
\newcommand{\bea}{\begin{eqnarray}}
\newcommand{\eea}{\end{eqnarray}}
\newcommand{\ba}{\begin{array}}
\newcommand{\ea}{\end{array}}

\newenvironment{frcseries}{\fontfamily{frc}\selectfont}{}

\def\sb {{\nabla}}
\def\LC{{\nabla^g}}

\begin{document}
\begin{center}
\vspace*{-1.0cm}

\vspace{1.0cm} {\Large \bf The Riemannian Bianchi identities  of metric connections\\[2mm] with  skew  torsion and generalized Ricci solitons}
 \\[.2cm]

\vspace{0.5cm}
 {\large S.~ Ivanov${}^1$ and  N. Stanchev$^2$}

\vspace{0.5cm}

${}^1$ University of Sofia, Faculty of Mathematics and
Informatics,\\ blvd. James Bourchier 5, 1164, Sofia, Bulgaria
\\and  Institute of Mathematics and Informatics,
Bulgarian Academy of Sciences\\
email: ivanovsp@fmi.uni-sofia.bg

\vspace{0.5cm}
${}^2$ University of Sofia, Faculty of Mathematics and
Informatics,\\ blvd. James Bourchier 5, 1164, Sofia, Bulgaria\\

\vspace{0.5cm}

\end{center}

\begin{abstract}
Curvature properties of a metric connection with totally skew-symmetric torsion are investigated. It is shown that if either the 3-form $T$ is harmonic, $dT=\delta T=0$ or the curvature of the torsion connection $R\in S^2\Lambda^2$ then the scalar curvature of a $\sb$-Einstein manifold is determined by the norm of the torsion up to a constant. It is proved that a compact generalized gradient Ricci soliton with closed torsion is Ricci flat if and only if  either the norm of the torsion or the Riemannian scalar curvature are  constants. In this case, the   torsion 3-form  is harmonic and the gradient function has to be constant.

Necessary and sufficient conditions for  a metric  connection with skew torsion to satisfy the Riemannian first Bianchi identity as well as the contracted Riemannian second Binachi identity  are presented. It is shown that if the torsion connection satisfies the Riemannian first Bianchi identity then it  satisfies the contracted Riemannian second Bianchi identity. It is also proved that a metric connection with skew torsion satisfying the curvature identity $R(X,Y,Z,V)=R(Z,Y,X,V)$ must be flat. 

\medskip


\medskip

\end{abstract}

\vskip 0.5cm
\noindent{\bf Acknowledgements:} \vskip 0.1cm
 We would like to thank Jeffrey Streets and Ilka Agricola for  extremely useful remarks, comments and suggestions.

The research of S.I.  is partially  financed by the European Union-Next Generation EU, through the National Recovery and Resilience Plan of the Republic of Bulgaria, project SUMMIT: BG-RRP-2.004-0008-C01. The research of N.S.  is partially supported  by Contract KP-06-H72-1/05.12.2023 with the National Science Fund of Bulgaria,  by Contract 80-10-192 / 17.5.2023   with the Sofia University "St.Kl.Ohridski" and  the National Science Fund of Bulgaria, National Scientific Program ``VIHREN", Project KP-06-DV-7.

\vskip 0.5cm

Conflict of interest: No conflict of interest.

\vskip 0.5cm

Data availability: applicable


\tableofcontents

\setcounter{section}{0}
\setcounter{subsection}{0}




\newsection{Introduction}

Riemannian manifolds with metric connections having totally skew-symmetric torsion and special holonomy received a strong interest mainly from supersymmetric string theories and supergravity.  The main interest is the existence and properties of such a connection with holonomy inside $U(n)$ which also  preserves the $U(n)$ structure i.e. metric connections with skew-symmetric torsion on an almost Hermitian manifold preserving the almost Hermitian structure.

Hermitian manifolds have widespread applications in both physics
and differential geometry. These are complex manifolds equipped
with a metric $g$, 
and a Kaehler form $\omega$ which  are of type  (1,1) with respect to the
complex structure $J$. There are many examples of Hermitian
manifolds as every complex manifold admits a Hermitian structure.
In many applications, Hermitian manifolds have  additional
properties which are expressed as either a condition  on $\omega$
or as a restriction on the holonomy of one of the Hermitian
connections.

Given a Hermitian manifold $(M,g,J)$  the  Strominger-Bismut connection $\nabla$ is the unique
connection on $M$  that is Hermitian ($\nabla g= 0$ and $\nabla J = 0$) and has
totally skew-symmetric torsion tensor $T$. Its existence and explicit expression first appeared in Strominger's seminal paper \cite{Str} in 1986 in connection with the heterotic supersymmetric string background, where he called it the H-connection. Three years later, Bismut formally discussed
and used this connection in his local index theorem paper \cite{bismut}, which leads to the name Bismut
connection in literature. Note that this connection was also called KT connection (K\"ahler with torsion) and characteristic connection. 

If the torsion 3-form $T$ of $\nabla$ is closed, $dT=0$, which is equivalent  to the condition $\partial\bar\partial\omega=0$, the Hermitian metric g is called SKT (strong K\"ahler with torsion) \cite{hkt} or pluriclosed. 
The SKT (pluriclosed) metrics have found many applications in
both physics, see e.g. \cite{hull, howe, Str, hethor,
hethor1,sethi} and geometry, see e.g. \cite{yaujost,IP1, IP2, salamon, swann, yau, FIUV,FT1, streets, liu}. For example in type II string theory, the torsion 3-form 
$T$ is identified with the 3-form field strength. This is required
by construction to satisfy $dT=0$. Streets and Tian
\cite{streets} introduced a Hermitian Ricci flow under which the
pluriclosed or equivalently strong KT structure is preserved. Generalizations of the pluriclosed condition $\partial\bar\partial\omega=0$ on 2n dimensional Hermitian manifolds in the form $
\omega^\ell\wedge
\partial\bar\partial \omega^k=0~, ~~~1\leq k+\ell\leq n-1~ $  have  been investigated in \cite{FU,Pop,FWW,IP3} etc.

Hermitian metrics with the Strominger-Bismut connection being K\"ahler-like, 
namely, its curvature satisfies the Riemannian first Bianchi identity ( the identity \eqref{RB} below), have been studied in \cite{AOUV},
investigating this property on 6-dimensional solvmanifolds with a  holomorphically
trivial canonical bundle. It was  conjectured by Angella-Otal-Ugarte-Villakampa in \cite{AOUV} that such  metrics should be SKT (pluriclosed).

Support to that can be derived from  [Theorem~5]\cite{WYZ} where the simply connected Hermitian manifolds with flat  Strominger-Bismut connection are classified and it easily follows that these spaces are  
SKT (pluriclosed) with parallel 3-form torsion \cite{AOUV}. The latter  conclusion also follows from the Cartan-Schouten theorem \cite{CS}, (see also [Theorem~2.2]\cite{AF}). 

The conjecture by  Angella-Otal-Ugarte-Villakampa was proved in \cite{ZZ}. In the first version of \cite{ZZ} in the arXive the authors stated,  by a misprint, that the  Strominger-Bismut connection is K\"ahler like if and only if the curvature satisfies the  identity 
 \begin{equation}\label{zz}R(X,Y,Z,V)=R(Z,Y,X,V).
 \end{equation}
It is shown in \cite{ZZ} that if the curvature of a Strominger-Bismut connection is K\"ahler like then one has $dT=\nabla T=0$. In fact, in the proof, they used the correct Riemannian first Bianchi identity  \eqref{RB}. 
The corrected definition of the K\"ahler-like condition was given in the second version of \cite{ZZ}.

 It was shown by Fino and Tardini in \cite{FTar} that the curvature condition \eqref{zz} is strictly stronger than the Riemannian first Bianchi identity by constructing an explicit example whose Strominger-Bismut curvature satisfies the Riemannian first Bianchi identity but it does not obey the condition \eqref{zz}.

In general, metric connections with skew-symmetric and closed torsion $T, dT=0$, are closely connected with the generalized Ricci flow. Namely, the fixed points of the generalized Ricci flow are Ricci flat metric connections with harmonic torsion 3-form, $Ric=dT=\delta T=0$, we refer to the recent book \cite{GFS} and the references given there for  mathematical and physical motivation. 

The first goal in the paper is to investigate properties of a metric connection with skew-symmetric torsion with applications to generalized Ricci flow.  We  consider $\sb$-Einstein spaces determined  with the condition  that the symmetric part of the Ricci tensor of the torsion connection  is a scalar multiple of the metric. These spaces are introduced by Agricola and Ferreira  \cite[Definition~2.2]{AFer} as the critical points of the $\sb$-Einstein-Hilbert functional   
 \cite[Theorem~2.1]{AFer}. In the case when the torsion 3-form is $\sb$-parallel the  $\sb$-Einstein spaces are investigated in \cite{AFer,AFF,Ch1,Ch2} and a large number of examples are given there. The  $\sb$-Einstein spaces appear also in generalized geometry. In  \cite[Proposition~3.47]{GFS} it is shown  that a Riemannian metric $g$ and a harmonic 3-form $T$ are  a critical point of the generalized Einstein-Hilbert functional if and only if it is $\sb$-Einstein. 

Notice that in contrast to the Riemannian case, for a $\sb$-Einstein manifold the scalar curvature $Scal$ of the torsion connection is not necessarily constant (for details see \cite{AFer}). If the torsion is $\sb$-parallel  then the scalar curvature of the torsion connection and the Riemannian scalar curvature are constants, similarly to an Einstein manifold  \cite[Proposition~2.7]{AFer}.

We investigate the constancy of the scalar curvatures of a $\sb$-Einstein space. We show in Theorem~\ref{mainE} that  if either the 3-form $T$  has zero torsion 1-forms, $\delta T\lrcorner T=T\lrcorner dT=0$, (in particular if $T$ is harmonic, $dT=\delta T=0$) or the curvature of the torsion connection $R\in S^2\Lambda^2$ then the scalar curvature of a $\sb$-Einstein manifold is determined by the norm of the torsion up to a constant. In particular, the scalar curvature of the torsion connection is constant if and only if  the norm of the torsion is constant.

Observing that if the torsion is parallel, $\sb T=0$, then the torsion 1-forms vanish, $\theta=\delta T\lrcorner T=0, \Theta=T\lrcorner dT=0$, we obtain that any $\sb$-Einstein manifold with parallel torsion of dimension bigger than 2 the Riemannian scalar curvature and the scalar curvature of the torsion connections are constants, thus confirm the result of Agricola and Ferreira, \cite[Proposition~2.7]{AFer}.

A special phenomenon occurs in dimension six, namely \eqref{ein9} below  imply
that if $(M,g,T)$  is a six-dimensional Riemannian manifold with zero torsion 1-forms (in particular, if $T$ is a harmonic 3-form)  and the  metric connection with  torsion $T$ is $\sb$-Einstein then the Riemannian scalar curvature is constant, Corollary~\ref{six}.

We recall \cite[Definition~4.31]{GFS} that a Riemannian manifold $(M,g,T)$ with a closed 3-form $T$ is a generalized  steady Ricci soliton with $k=0$ if one has 
\begin{equation*}
Ric^g=\frac14T^2-\frac12\mathbb{L}_Xg, \qquad \delta T=-X\lrcorner T, \qquad dT=0.
\end{equation*}
If the vector field $X$ is a gradient of a smooth function $f$ then we have the notion of a generalized gradient Ricci soliton. A complete generalized gradient Ricci solitons are constructed on complex surfaces  in  \cite{SU}. The existence and classification of non-trivial solitons on compact (complex) 4-manifolds has been recently proved in \cite{streets1,StrUs,ASU}.

Our first main observation is the following 
\begin{thrm}\label{mainR}
Let $(M,g,T)$ be a compact Riemannian manifold with closed 3-form $T, dT=0$. 

If $(M,g,T,f)$ is a generalized gradient Ricci soliton then 
the following conditions are equivalent:
\begin{itemize}
\item[a).] The norm of the torsion is constant, $d||T||^2=0$;
\item[b).] The function $f$ is constant.
\item[c).] The torsion connection is Ricci flat, $Ric=0$;
\item[d).] The Riemannian scalar curvature is constant, $Scal^g=const$.
\end{itemize}
In all four cases, the 3-form $T$ is harmonic.
\end{thrm}
\begin{rmrk}
Theorem~\ref{mainR} shows, in particular, that there are no homogeneous compact generalized gradient Ricci solitons with a non-trivial gradient vector field.
\end{rmrk}

Examples of Ricci flat torsion connections are constructed in \cite{PR,LW}.

The second  aim of this  note is to express necessary and sufficient conditions the curvature of a metric connection with totally skew-symmetric torsion to satisfy the Riemannian first Bianchi identity as well as the contracted Riemannian second Binachi identity (\eqref{RB2} below)  and equation \eqref{zz}. 

Our second main result is the next
 \begin{thrm}\label{th1} 
 The curvature of a   metric connection $\nabla$ with skew-symmetric torsion $T$ on a Riemannian manifold $(M,g)$ satisfies the Riemannian first Bianchi identity if and only if  the next identities hold
 \begin{equation}\label{bsk}
 dT=-2\nabla T=\frac23\sigma^T,
 \end{equation}
 where  the four form $\sigma ^T$ corresponding to the 3-form $T$ is defined below  with \eqref{sigma}.
 
In this case,  the norm of the 3-form $T$ is a constant, $||T||^2=const.$ and the curvature of the connection $\sb$ satisfies the contracted Riemannian second Bianchi identity.
 \end{thrm}
Clearly, any torsion-free connection satisfying \eqref{zz} must be flat which is  a simple consequence of the first Bianchi identity. In particular, a Riemannian manifold satisfies \eqref{zz} if and only if it is flat. 

We show that this is valid  also for metric connections with skew-symmetric torsion which  explains the reason for the existence of  the example in \cite{FTar}. We derive from Theorem~\ref{th1}  the next general result
 \begin{thrm}\label{thzz}
 A metric connection with skew-symmetric torsion satisfies the  condition \eqref{zz} if and only if  it is flat, $R=0$.
 \end{thrm}
 
 \begin{rmrk}
We note that  Theorem~\ref{mainE} generalizes  the Ricci flat case established  recently in \cite[Lemma~2.21]{Lee} where it was proved  that if the torsion is harmonic (it is sufficient only to be closed) and the Ricci tensor of the torsion connection vanishes then the norm of the torsion is constant. 

The converse is not true in general. Namely not any space with closed torsion of constant norm $dT=d||T||^2=0$  is Ricci flat. 

In some cases the converse holds true. For example, in the case of compact generalized gradient Ricci soliton the Ricci flatness is equivalent to the constancy of the norm of the torsion due to Theorem~\ref{mainR}. 

Another cases occur if the torsion connection has special holonomy, contained  in the groups $SU(3)$, $G_2$  or $Spin(7)$.  It is shown very recently in \cite {IS1}, \cite{IS2}  and \cite{IP} that  the $SU(3), G_2, Spin(7)$-torsion connection with closed torsion is Ricci flat on a compact manifold if and only if the norm of the torsion is constant.
 \end{rmrk}

Everywhere in the paper, we will make no difference between tensors and the corresponding forms via the metric as well as we will  use Einstein summation conventions, i.e. repeated Latin  indices are summed over up to $n$.

\section{Metric connection with skew-symmetric  torsion and its curvature}
On a Riemannian manifold $(M,g)$ of dimension $n$ any metric connection $\sb$ with totally skew-symmetric torsion $T$ is connected with the Levi-Civita connection $\sb^g$ of the metric $g$ by
\begin{equation}\label{tsym}
\sb^g=\sb- \frac12T.
\end{equation}
The exterior derivative $dT$ has the following  expression (see e.g. \cite{I,IP2,FI})
\begin{equation}\label{dh}
\begin{split}
dT(X,Y,Z,V)=(\nabla_XT)(Y,Z,V)+(\nabla_YT)(Z,X,V)+(\nabla_ZT)(X,Y,V)\\+2\sigma^T(X,Y,Z,V)-(\nabla_VT)(X,Y,Z),
 \end{split}
 \end{equation}
where the 4-form $\sigma^T$, introduced in \cite{FI}, is defined by
 \begin{equation}\label{sigma}
 \sigma ^T(X,Y,Z,V)=\frac12\sum_{j=1}^n(e_j\lrcorner T)\wedge(e_j\lrcorner T)(X,Y,Z,V) ,
\end{equation} 
$(e_j\lrcorner T)(X,Y)=T(e_j,X,Y)$ is the interior multiplication and $\{e_1,\dots,e_n\}$ is an orthonormal  basis.

The properties of the 4-form $\sigma^T$ are studied in detail in \cite{AFF} where it is shown that $\sigma^T$ measures the `degeneracy' of the 3-form $T$.

 For the curvature of $\sb$ we use the convention $ R(X,Y)Z=[\nabla_X,\nabla_Y]Z -
 \nabla_{[X,Y]}Z$ and $ R(X,Y,Z,V)=g(R(X,Y)Z,V)$. It has the well-known properties
 \begin{equation}\label{r1}
 R(X,Y,Z,V)=-R(Y,X,Z,V)=-R(X,Y,V,Z).
 \end{equation}
 The   Ricci tensors and scalar curvatures of the Levi-Civita connection $\LC$ and the torsion connection $\sb$ are related by \cite[Section~2]{FI}, (see also \cite [Prop. 3.18]{GFS})
\begin{equation}\label{rics}
\begin{split}
Ric^g(X,Y)=Ric(X,Y)+\frac12 (\delta T)(X,Y)+\frac14\sum_{i=1}^ng(T(X,e_i),T(Y,e_i);\\
Scal^g=Scal+\frac14||T||^2,\qquad Ric(X,Y)-Ric(Y,X)=-(\delta T)(X,Y),
\end{split}
\end{equation}
where $\delta=(-1)^{np+n+1}*d*$ is the co-differential acting on $p$-forms and $*$ is the Hodge star operator.

Following \cite{GFS} we denote 
$$T^2_{ij}=T_{iab}T_{jab}:=\sum_{a,b=1}^nT_{iab}T_{jab}.$$
 Then the first equality in \eqref{rics} takes the form
\[Ric^g_{ij}=Ric_{ij}+\frac12\delta T_{ij}+\frac14T^2_{ij}.\]
 The first Bianchi identity for $\nabla$ 
 \begin{equation*}
 \begin{split}
 R(X,Y,Z,V)+ R(Y,Z,X,V)+ R(Z,X,Y,V)\\=(\nabla_XT)(Y,Z,V)+(\nabla_YT)(Z,X,V)+(\nabla_ZT)(X,Y,V)+\sigma^T(X,Y,Z,V)
 \end{split}
 \end{equation*}
   can be written in the following form (see e.g. \cite{I,IP2,FI})
 \begin{equation}\label{1bi}
 \begin{split}
 R(X,Y,Z,V)+ R(Y,Z,X,V)+ R(Z,X,Y,V)\\
 =dT(X,Y,Z,V)-\sigma^T(X,Y,Z,V)+(\nabla_VT)(X,Y,Z).
 \end{split}
 \end{equation}
It is proved in \cite[p.307]{FI} that the curvature of  a metric connection $\sb$ with totally skew-symmetric torsion $T$  satisfies also the next identity
 \begin{equation}\label{gen}
 \begin{split}
 R(X,Y,Z,V)+ R(Y,Z,X,V)+ R(Z,X,Y,V)-R(V,X,Y,Z)-R(V,Y,Z,X)-R(V,Z,X,Y)\\
 =\frac32dT(X,Y,Z,V)-\sigma^T(X,Y,Z,V).
 \end{split}
 \end{equation}
 We obtain from \eqref{gen} and \eqref{1bi} 
 \begin{prop}
 The curvature of a metric connection with skew-symmetric torsion satisfies
 \begin{equation}\label{1bi1}
 \begin{split}
R(V,X,Y,Z)+R(V,Y,Z,X)+R(V,Z,X,Y)= -\frac12dT(X,Y,Z,V)+(\nabla_VT)(X,Y,Z).
 \end{split}
 \end{equation}
 \end{prop}
  Following \cite{AOUV} we have
\begin{dfn} We say that the curvature $R$ satisfies the Riemannian first Bianchi identity if 
\begin{equation}\label{RB}
R(X,Y,Z,V)+R(Y,Z,X,V)+R(Z,X,Y,V)=0.
\end{equation}
\end{dfn}
 It is well known algebraic fact that \eqref{r1} and \eqref{RB} imply $R\in S^2\Lambda^2$ (c.f. \cite[Chapter~5]{KN1}), i.e it holds
 \begin{equation}\label{r4}
 R(X,Y,Z,V)=R(Z,V,X,Y).
 \end{equation}
 Note that, in general, \eqref{r1} and \eqref{r4} do not imply \eqref{RB}.
 
The precise condition the curvature of 
a metric connection $\sb$ with totally skew-symmetric torsion $T$ to satify \eqref{r4}  is given in \cite[Lemma~3.4]{I}, namely   the covariant derivative of the torsion with respect to the torsion connection $\sb T$ must be  a four form,
\begin{lemma} \cite[Lemma~3.4]{I}) The next equivalences hold for a metric connection with torsion 3-form
\begin{equation}\label{fourf}
(\sb_XT)(Y,Z,V)=-(\sb_YT)(X,Z,V) \Longleftrightarrow R(X,Y,Z,V)=R(Z,V,X,Y) ) \Longleftrightarrow dT=4\LC T.
\end{equation}
\end{lemma}
An immediate consequence of \eqref{fourf} is the  next
\begin{cor}
Suppose that a  metric connection with  torsion 3-form $T$ has curvature $R\in S^2\Lambda^2$. 

Then the torsion 3-form is closed, $dT=0$, if and only if the torsion is parallel with respect to the Levi-Civita connection, $\LC T=0$.
\end{cor}
 \begin{dfn}
We say that  the curvature of a metric connection with skew-symmetric torsion satisfies the contracted Riemannian second Bianchi identity if 
\begin{equation}\label{RB2}
d(Scal)(X)-2\sum_{i=1}^n(\sb_{e_i}Ric)(X,e_i)=0.
\end{equation}
\end{dfn}
 Following \cite{AF}, we consider  a 1-parameter family of metric connections $\nabla^t$ with torsion $tT$ defined by
 \begin{equation}\nonumber g(\nabla^g_XY,Z)=g(\nabla^t_XY,Z)-\frac{t}2T(X,Y,Z)
 \end{equation}
 yielding the equality (see e.g. \cite{AF})
 \begin{equation}\label{h-gh}
(\nabla^g_XT)(Y,Z,V)=(\nabla^t_XT)(Y,Z,V)+\frac{t}2\sigma^T(X,Y,Z,V).
 \end{equation}
  We continue with the following  result which generalizes [Proposition~2.1]\cite{AF},  [Theorems~3.1,3.2]\cite{FTar} and proves the first part of Theorem~\ref{th1}.
 \begin{thrm}\label{th111}
 The curvature of a metric connection with skew-symmetric torsion $T$ satisfies the Riemannian first Bianchi identity if and only if the identities \eqref{bsk} hold.

 In this case the  3-form $T$ is parallel with respect to the metric connection $\nabla^{1/3}$ with torsion equal to $\frac13 T,  \nabla^{1/3}T=0$.
  In particular, the norm of the 3-form $T$ is constant, $||T||^2=const$.
 \end{thrm}
 \begin{proof}
 Let \eqref{bsk} hold. Then the covariant derivative $\nabla T$ is a four form. 
 Substitute \eqref{bsk} into the right-hand side of \eqref{1bi} to get that the Riemannian first Bianchi identity \eqref{RB} holds.

 For the converse,  suppose  the curvature $R$ of $\sb$ satisfies the Riemannian first Bianchi identity \eqref{RB}. We use the  well-known algebraic fact that the curvature identities \eqref{r1} and \eqref{RB} imply  the identity \eqref{r4} (see e.g. \cite[Chapter~5]{KN1}).  Now, \cite[Lemma~3.4]{I} tells us  that   $\nabla T$ is a 4-form. Then \eqref{dh}, \eqref{1bi} together with \eqref{RB}  yield $$dT=4\nabla T+2\sigma^T, \qquad dT=\nabla T+\sigma^T.$$
 The last two equalities 
imply \eqref{bsk} holds.

One gets from \eqref{bsk}  $\nabla T=-\frac13\sigma^T$ and \eqref{h-gh} yields $\nabla^{1/3}T=0$, which was first obsered in \cite{AF}.
  \end{proof}

 \subsection{Proof of Theorem~\ref{thzz}}

 \begin{proof}
  We will show that \eqref{zz} implies \eqref{bsk}. Indeed, using \eqref{zz} together with \eqref{r1} we have the next sequence of equalities
 $$-R(V,X,Y,Z)=R(X,V,Y,Z)=R(X,Y,V,Z)=-R(X,Y,Z,V).$$
  Apply the latter to \eqref{gen} to get
 \begin{equation}\label{zz1}3dT-2\sigma^T=0.
 \end{equation}
 Further, \eqref{zz} yields $R(X,Y,Z,V)=R(Z,V,X,Y)$ leading by \cite[Lemma~3.4]{I} that $\nabla T$ is a 4-form which applied to \eqref{dh} gives
\begin{equation*}
\begin{split}
dT(X,Y,Z,V)
=4\nabla T(X,Y,Z,V)+2\sigma^T(X,Y,Z,V)=\frac23\sigma^T(X,Y,Z,V),
 \end{split}
 \end{equation*}
 where we have used \eqref{zz1}. 
 
 Hence, \eqref{bsk} holds and Theorem~\ref{th111} shows the validity of the Riemannian first Bianchi identity \eqref{RB}. 
 Consequently, we  get
 \begin{equation*}
 \begin{split}
 0=R(X,Y,Z,V)+R(Y,Z,X,V)+R(Z,X,Y,V)\\=R(Z,Y,X,V)+R(Y,Z,X,V)+R(Z,X,Y,V)=R(Z,X,Y,V),
 \end{split}
 \end{equation*}
 where we have applied \eqref{zz} to conclude the second identity. The proof is completed.
 \end{proof}
\section{The contracted  second Bianchi identity}
In this section, we investigate the second Bianchi identity for the curvature  of a metric connection with a skew-symmetric  torsion. 

First, we show the validity of an algebraic identity.
\begin{prop}
For  an arbitrary 3-form $T$  the next identity holds
\begin{equation}\label{sigt}
 T_{abc}\sigma^T_{abci}=0.
\end{equation}
\end{prop}
\begin{proof}
Using  \eqref{sigma} we calculate
\[T_{abc}\sigma^T_{abci}=T_{abc}\Big(T_{abs}T_{sci}+T_{bcs}T_{sai}+T_{cas}T_{sbi} \Big)=3T_{abc}T_{abs}T_{sci}=0\]
since $T_{abc}T_{abs}$ is symmetric in $c$ and $s$ while $T_{sci}$ is skew-symmetric in $c$ and $s$.
\end{proof}
The next observation expresses the $\sb$-divergence of $\delta T$, 
\begin{prop}\label{genprop}
For a metric connection with torsion 3-form $T$, the next identity holds
\begin{equation}\label{ein10}
2\sb_i\delta T_{ij}
=\delta T_{ia}T_{iaj}.
\end{equation}
\end{prop}
\begin{proof}
Applying  \eqref{tsym}, we calculate using the Ricci identity for the Levi-Civita connection, the symmetricity of its Ricci tensor, and the first Bianchi identity 
\begin{multline*}\sb_i\delta T_{ij}=\LC_i\delta T_{ij}-\frac12\delta T_{is}T_{ijs}=-\frac12\Big(\LC_i\LC_s-\LC_s\LC_i\Big)T_{sij}+\frac12\delta T_{is}T_{isj}=\\
-\frac12\Big(R^g_{issq}T_{qij}+R^g_{isiq}T_{sqj} +R^g_{isjq}T_{isq} \Big) +\frac12\delta T_{is}T_{isj}\\=Ric^g_{iq}T_{iqj}-\frac16\Big(R^g_{isqj}+R^g_{siqj}+R^g_{iqsj}\Big)T_{isq}+\frac12\delta T_{is}T_{isj}=\frac12\delta T_{is}T_{isj}.
\end{multline*}
The proof of Proposition~\ref{genprop} is completed.
\end{proof}
\begin{dfn}
For any 3-form $T$ we define two torsion 1-forms $\theta,\Theta$ naturally associated to $T$ by
\begin{equation*}
\theta_j=\delta T_{ab}T_{jab},\quad \Theta_j=T_{abc}dT_{jabc}.
\end{equation*}
\end{dfn}
Further, we  have
\begin{lemma}\label{l11}
The  1-forms $\theta$ and $\Theta$ are connected by the equality
\begin{equation}\label{e13}
3\theta_j+\Theta_j=\frac12\LC_j||T||^2-3\LC_sT^2_{sj}=\frac12\sb_j||T||^2-3\sb_sT^2_{sj}.
\end{equation}
\end{lemma}
\begin{proof}
We obtain from \eqref{dh} 
\begin{equation}\label{e12}
\Theta_j=-dT_{abcj}T_{abc}=-3\LC_aT_{bcj}T_{abc}+\frac12\LC_j||T||^2=-3\sb_aT_{bcj}T_{abc}+\frac12\sb_j||T||^2, 
\end{equation}
 where we used \eqref{h-gh} and  \eqref{sigt}.
 
On the other hand, we calculate in view of \eqref{sigt} and \eqref{e12} that
\begin{multline}\label{ein5}
\theta_j=\delta T_{ia}T_{iaj}=-\sb_sT_{sia}T_{iaj}=-\sb_sT^2_{sj}+\sb_sT_{iaj}T_{sia}=-\LC_sT^2_{sj}+\LC_sT_{iaj}T_{sia}\\=
-\LC_sT^2_{sj}+\frac13(\LC_sT_{iaj}+\LC_iT_{asj}T_{sia}+\LC_aT_{sij})T_{sia}\\=-\LC_sT^2_{sj}-\frac13dT_{jsia}T_{sia}+\frac16\LC_j||T||^2=-\LC_sT^2_{sj}-\frac13\Theta_j+\frac16\LC_j||T||^2.
\end{multline}
The lemma is proved.
\end{proof}
Note that if the torsion is $\sb$-paralel then $\theta=\delta T\lrcorner T=0$ and $\Theta=T\lrcorner dT=2T\lrcorner\sigma^T=0$ due to \eqref{dh} and \eqref{sigt}.
\begin{prop}\label{2bif}
The contracted second Bianchi identity for the curvature of the torsion connection $\sb$ reads
\begin{equation}\label{e1}
d(Scal)(X)-2\sum_{i=1}^n(\sb_{e_i}Ric)(X,e_i)+\frac16d||T||^2(X)+\theta(X)+\frac16\Theta(X)=0.
\end{equation}
If the torsion 1-forms satisfy the identity
\begin{equation*}
6\theta +\Theta=0
\end{equation*} 
then
\begin{equation}\label{eqdT}
d(Scal)_j-2\sb_iRic_{ji}+\frac16\sb_j||T||^2=0.
\end{equation}
In particular, if the 3-form $T$ is harmonic, $dT=\delta T=0$ then \eqref{eqdT} holds.
\end{prop}
\begin{proof}
The second Bianchi identity for  the curvature of a metric connection $\sb$ with torsion $T$ is
\begin{multline}\label{secB}
(\sb_XR)(V,Y,Z,W)+(\sb_VR)(Y,X,Z,W)+(\sb_YR)(X,V,Z,W)\\+R(T(X,V),Y,Z,W)+R(T(V,Y),X,Z,W)+R(T(Y,X),V,Z,W)=0.
\end{multline}
Take the trace of \eqref{secB}
 to get 
\begin{equation}\label{secric}
\begin{split}
(\sb_XRic)(Y,Z)+\sum_{i=1}^n(\sb_{e_i}R)(Y,X,Z,e_i)-(\sb_YRic)(X,Z)\\
+\sum_{i,j=1}^n\Big[T(X,e_i,e_j)R(e_j,Y,Z,e_i) +T(e_i,Y,e_j)R(e_j,X,Z,e_i)\Big]-\sum_{i=1}^nT(Y,X,e_i)Ric(e_i,Z)=0.
\end{split}
\end{equation}
The trace in \eqref{secric} together with \eqref{rics} and \eqref{1bi1}  yields
\begin{equation}\label{rt}
\begin{split}
0=d(Scal)(X)-2\sum_{i=1}^n(\sb_{e_i}Ric)(X,e_i)-2\sum_{i,j=1}^nT(X,e_i,e_j)Ric(e_i,e_j)\\+\frac13\sum_{i,j,k=1}^nT(e_i,e_j,e_k)\Big[R(X,e_i,e_j,e_k)+R(X,e_j,e_k,e_i)+R(X,e_k,e_i,e_j)\Big]\\
=d(Scal)X)-2\sum_{i=1}^n(\sb_{e_i}Ric)(X,e_i)+\sum_{i,j=1}^nT(X,e_i,e_j)\delta T(e_i,e_j)+\frac16\sb_X||T||^2\\+\frac16\sum_{i,j,k=1}^nT(e_i,e_j,e_k)dT(X,e_i,e_j,e_k).
\end{split}
\end{equation}
which proves \eqref{e1}. The proof of the Proposition~\ref{2bif} is completed.
\end{proof}
Apply \eqref{e12} to the last term of \eqref{e1} to get
\begin{cor}
The contracted second Bianchi identity for  the curvature of the torsion connection has also the form
\begin{equation}\label{2bi}
d(Scal)_j-2\sb_sRic_{js}+\frac14\sb_j||T||^2+\theta_j-\frac12\sb_iT_{jka}T_{ijk}=0.
\end{equation}
\end{cor}
 \begin{cor}
 The   curvature of the torsion connection satisfies the contracted Riemannian second Bianchi identity \eqref{RB2} if and only if the next equality holds
\begin{equation}\label{biii}
6\theta_j+\Theta_j+\sb_j||T||^2=0\Longleftrightarrow 4\theta_j+\sb_j||T||^2-2\sb_iT_{jka}T_{ijk}=0.
\end{equation} 
In particular, the equation \eqref{biii} holds for any Ricci flat torsion connection.
 \end{cor}
 Note that a special case of Proposition~\ref{2bif} and the above corollaries,  when the torsion is $\sb$-parallel, is given in \cite[Corollary~2.6]{AFer}.
\begin{thrm}\label{s2l2}
Let  the curvature $R$ of a  metric connection $\sb$ with skew-symmetric torsion $T$  satisfies $R\in S^2\Lambda^2$, i.e. \eqref{r4} holds. 

Then   the curvature of   $\sb$ satisfies the contracted Riemannian second Bianchi identity \eqref{RB2} if and only if the norm of the torsion is a constant, $||T||^2=const$.

In particular, if $\sb$ is Ricci flat, $Ric=0$, then the norm of the torsion  is  constant, $||T||=const.$
\end{thrm}
\begin{proof}
The condition \eqref{r4}  implies the Ricci tensor is symmetric and \eqref{rics} yields $\delta T=0$.  Now, taking into account \eqref{fourf} we obtain from \eqref{2bi} 
\begin{equation*}\label{2biff}
d(Scal)(X)-2\sum_{i=1}^n(\sb_{e_i}Ric)(X,e_i)+\frac12\sb_X||T||^2=0
\end{equation*}
which completes the proof of Theorem~\ref{s2l2} since any Ricci flat connection satisfies the contracted Rieamannian second Bianchi identity.
\end{proof}
Theorem~\ref{th111} and Theorem~\ref{s2l2} imply
\begin{cor}\label{co1}
If the curvature of the torsion connection satisfies the Riemannian first Bianchi identity \eqref{RB} then it satisfies the contracted Riemannian  second Bianchi identity \eqref{RB2}.
\end{cor}
\subsection{Proof of Theorem~\ref{th1}}
The proof of Theorem~\ref{th1} follows from Theorem~\ref{th111} and Corollary~\ref{co1}.

\begin{rmrk}
It is known that if the curvature of a metric connection satisfies the Riemannian first Bianchi identity \eqref{RB}  then it satisfies the curvature identity \eqref{r4} but the converse is not true in general. 

In some cases the converse is  true. If the torsion connection has a special holonomy, contained, for example in the Lie group $SU(3)$ in dimension six,  or in the exeptional Lie group $G_2$ in dimension seven the first Bianchi identity implies the vanishing of the Ricci tensor, $Ric=0$. It is shown  recently in \cite {IS1} and \cite{IS2} that \eqref{r4} together with $Ric=0$ for a torsion connection with holonomy contained in $SU(3)$ or in $G_2$ imply  that the Riemannian  first Bianchi identity \eqref{RB} follows.
\end{rmrk}

\section{The $\sb$-Einstein condition}
Since the Ricci tensor of the torsion connection is not symmetric, the usual Einstein condition seems to be restrictive. We consider the following weaker condition introduced by Agricola and Ferreira \cite{AFer}
\begin{dfn}\cite[Definition~2.2]{AFer}
A metric connection with skew-symmetric torsion is said to be $\sb$-Einsten if the symmetric part of the Ricci tensor is a scalar multiple of the metric,
\begin{equation*}
\frac{Ric(X,Y)+Ric(Y,X)}2=\lambda g(X,Y).
\end{equation*}
\end{dfn}
In view of \eqref{rics} the $\sb$-Einstein condition is equivalent to 
\begin{equation}\label{ein2}
Ric(X,Y)=\frac{Scal}{n}g(X,Y)-\frac12\delta T(X,Y).
\end{equation}
It is shown in \cite[Theorem~2.1]{AFer} that on a compact Riemanian manifold $(M,g,T)$ with a Riemannian metic $g$ and a 3-form $T$, the critical points of the $\sb$-Einstein-Hilbert functional 
\[\mathcal{L}(g,T)=\int_MScal. Vol_g\]
are precisely the pairs $(g,T)$ satisfying the $\sb$-Einstein condition.

We have
\begin{thrm}\label{mainE}
Let a metric connection with skew-symmetric torsion $T$ be  $\sb$-Einstein.
\begin{itemize}
\item[a)] Then  the next identity holds
\begin{equation}\label{ein6}
\frac{n-2}{n}d(Scal)_j-\frac12\LC_sT^2_{sj}+\frac14\sb_j||T||^2=0.
\end{equation}
\item[b)]  If the torsion 1-forms satisfy the identity
\begin{equation}\label{t11}
3\theta+\Theta=0.
\end{equation}
then the scalar curvature is determined by the norm of the torsion up to a constant $C$ due to 
\begin{equation}\label{ein9}Scal =-\frac{n}{6(n-2)}||T||^2+C\quad  and \quad Scal^g=Scal+\frac14||T||^2=\frac{n-6}{12(n-2)}||T||^2+C.
  \end{equation}
  In particulart, if the 3-form $T$ is harmonic, $dT=\delta T=0$ then \eqref{ein9} holds.
\item[c)] If the curvature of the torsion connection $R\in S^2\Lambda^2$ then the scalar curvature and  the norm of the torsion satisfy the next identity with  a constant $B$ 
 \begin{equation}\label{ein8}
Scal=-\frac{n}{2(n-2)}||T||^2+B,\quad Scal^g=-\frac{n+2}{4(n-2)}||T||^2+B.
\end{equation}
In particular, if the  scalar curvature of the torsion connection is constant in b) and c)  then the norm of the torsion is constant, $||T||^2=const$.
\end{itemize}
\end{thrm}
\begin{proof}
We have from \eqref{ein2} and \eqref{ein10} that 
\begin{equation}\label{ein3}
\sb_iRic_{ji}=\frac{d(Scal)_j}{n}-\frac12\sb_i\delta T_{ji}=
\frac{d(Scal)_j}{n}+\frac14\delta T_{ia}T_{iaj}=\frac{d(Scal)_j}{n}+\frac14\theta_j.
\end{equation}

Substitute \eqref{ein3} and \eqref{ein5}  into \eqref{rt} to get
\begin{multline}\label{ein11}
0=\frac{n-2}{n}d(Scal)_j+\frac12\delta T_{ia}T_{iaj}+\frac16\sb_j||T||^2+\frac16T_{abc}dT_{jabc}\\=\frac{n-2}{n}d(Scal)_j-\frac12\LC_sT^2_{sj}+\frac14\sb_j||T||^2
\end{multline}
since $\sb_j||T||^2=\LC_j||T||^2$ due to \eqref{sigt}. This proves \eqref{ein6}.

If $3\theta+\Theta=0$ then \eqref{ein11} takes the form
\begin{equation}\label{ein7}
0=d(Scal)_j-2\sb_iRic_{ij}+\frac16\sb_j||T||^2=d(\frac{n-2}{n}Scal+\frac16||T||^2)_j.
\end{equation}
  Hence, \eqref{ein9} holds.

 If  $R\in S^2\Lambda^2$ then $\sb T$ is a four form due to \cite[Lemma~3.4]{I}. Consequently, we have $\delta T=0$ and $\sb_sT^2_{sj}=-\frac12\sb_j||T||^2$ because of  \eqref{ein5}. Now, \eqref{ein11} takes the form
\[
0=d(\frac{n-2}{n}Scal+\frac12||T||^2),
\]
which implies \eqref{ein8}. The proof is completed.
\end{proof}
We remark that due to \eqref{e13} the condition \eqref{t11} in Theorem~\ref{mainE} is equivalent to the condition
\[3\theta+\Theta=0 \Longleftrightarrow  \sb_j||T||^2=6\sb_iT^2_{ij}.\]
The second equality in \eqref{ein9} leads to the next
\begin{cor}\label{six} Let 
 $(M,g,T)$  be a six dimensional Riemannian manifold and  the 3-form $T$ satisfies \eqref{t11}, (in particular $T$ be a harmonic 3-form). If the  metric connection with  torsion $T$ is $\sb$-Einstein then the Riemannian scalar curvature is constant.
\end{cor}
\begin{rmrk}
Note that if the torsion is harmonic then the last equality in \eqref{ein7} follows from \cite[Proposition~3.47]{GFS}.
\end{rmrk}
We also remark that if $(M,g,T)$ is Ricci flat with  closed torsion  then the  equality \eqref{ein9} yields the norm of the torsion is constant which recovers the recent result  \cite[Lemma~2.21]{Lee}.
\section{Generalized gradient Ricci solitons. Proof of Theorem~\ref{mainR}}
One fundamental consequence of Perelman's energy formula for Ricci flow is that compact steady solitons for Ricci flow are automatically gradient.  Adapting these energy functionals to generalized Ricci
flow it is proved in  \cite[Chapter~6]{GFS}  that steady generalized Ricci solitons on compact manifolds are automatically gradient, and moreover,  satisfy k = 0. 

We recal \cite[Definition~4.31]{GFS} that a Riemannian manifold $(M,g,T,f )$ with a closed 3-form $T$ and a smooth function $f$ is a generalized gradient Ricci soliton with $k=0$ if one has 
\begin{equation}\label{gein1}
Ric^g_{ij}=\frac14T^2_{ij}-\LC_i\LC_j f, \qquad \delta T_{ij}=-df_sT_{sij}, \qquad dT=0.
\end{equation}
Using the torsion  connection $\sb$ with 3-form torsion $T$, \eqref{tsym} and the second equation in\eqref{gein1}, we have 
\begin{equation}\label{gein2}\sb_i\sb_jf-\sb_j\sb_if=-df_sT_{sij}=\delta T_{ij}.
\end{equation}
In view of  \eqref{rics} and \eqref{gein2}  we write \eqref{gein1} in the form
\begin{equation}\label{gein3}
Ric_{ij}=-\frac12(\sb_i\sb_jf+\sb_j\sb_if)-\frac12\delta T_{ij}=-\sb_i\sb_j f, \quad Scal =-\sb_i\sb_if=\Delta f.
\end{equation}
The second Bianchi identity \eqref{rt} and \eqref{gein3}
yield
\begin{equation}\label{gein4}
\sb_j\Delta f-2\sb_iRic_{ji}+\delta T_{ab}T_{abj}+\frac16\sb_j||T||^2=\sb_j\Delta f+2\sb_i\sb_j\sb_i f+\delta T_{ab}T_{abj}+\frac16\sb_j||T||^2=0
\end{equation}
We evaluate the second term of \eqref{gein4} in two ways. First using the Ricci identities for $\sb$ and \eqref{gein2} 
\begin{multline}\label{gein5}
\sb_i\sb_j\sb_if=\sb_j\sb_i\sb_if-R_{ijis}\sb_sf-T_{ija}\sb_a\sb_i f=-\sb_j\Delta f+Ris_{js}\sb_sf-\frac12(\sb_a\sb_if-\sb_i\sb_af)T_{aij}\\=-\sb_j\Delta f-\sb_j\sb_sf.\sb_sf+\frac12df_sT_{sai}T_{aij}=-\sb_j\Delta f-\sb_j\sb_sf.\sb_sf-\frac12\delta T_{ai}T_{aij}
\end{multline}
Applying \eqref{gein2} and \eqref{ein10}, we obtain
\begin{equation}\label{gein6}
\sb_i\sb_j\sb_if=\sb_i(\sb_i\sb_jf-\delta T_{ij})=\sb_i\sb_i\sb_jf-\sb_i\delta T_{ij}=\sb_i\sb_i\sb_jf-\frac12\delta T_{ia}T_{iaj}
\end{equation}
Substitute \eqref{gein5} and \eqref{gein6} into \eqref{gein4}, we get
\begin{equation}\label{gein7}
\begin{split}
-\sb_j\Delta f-\sb_j||df||^2+\frac16\sb_j||T||^2=-\sb_j\Delta f+2Ric_{js}\sb_sf+\frac16\sb_j||T||^2=0;\\
\sb_j\Delta f+2\sb_i\sb_i\sb_j f+\frac16\sb_j||T||^2=\sb_j\Delta f-2\sb_iRic_{ij}+\frac16\sb_j||T||^2=0.
\end{split}
\end{equation}
Note that the first equality in \eqref{gein7} is precisely  \cite[Proposition~4.33]{GFS}.

Differentiate the first line in \eqref{gein7} apply \eqref{gein3} and the second equality in \eqref{gein7} to get
\begin{multline}\label{gein8}
0=\Delta(\Delta f-\frac16||T||^2)+2\sb_jRic_{js}\sb_sf+2Ric_{js}\sb_j\sb_sf=\Delta(\Delta f-\frac16||T||^2)+2\sb_jRic_{js}\sb_sf-2||Ric||^2\\
=\Delta(\Delta f-\frac16||T||^2)+\sb_j(\Delta f+\frac16||T||^2)\sb_jf-2||Ric||^2\\\le\Delta(\Delta f-\frac16||T||^2)+g(\sb(\Delta f +\frac16||T||^2),\sb f)
\end{multline}
If the norm of the torsion $T$ is constant, $\sb||T||^2=0$, then  \eqref{gein8} takes the form
\[\Delta\Delta f+g(\sb \Delta f,\sb f)\ge 0.\]
If $M$ is compact, $\Delta f$ is constant due to the strong maximum principle (see e.g. \cite{YB,GFS}) which yields $f= const$. Conversely, if the function $f$ is constant then \eqref{gein8}  together with the strong maximum principle implies $d||T||^2=0$ which yields the equivalence of a) and b). 

Assume a) or b). Then $Ric=Scal=\delta T=0$ due to \eqref{gein2} and \eqref{gein3}. If $Ric=0$ then $\Delta f=0$ leading to $f=const$ since $M$ is compact. Hence, b) is equivalent to c).

To show d) is equivalent to a) we use \eqref{rics} and \eqref{gein3} to find $Scal^g=\Delta f+\frac14||T||^2$ and we can write \eqref{gein8} in the form
\begin{equation}\label{ffinn}0\le\Delta(Scal^g-\frac5{12}||T||^2)+g(\sb(Scal^g-\frac1{12}||T||^2),\sb f).
\end{equation}
Then $Scal^g=const.$ if and only $d||T||^2=0$ by the strong maximum principle applied to \eqref{ffinn}.

The proof of Theorem~\ref{mainR} is completed.
\begin{rmrk}
Theorem~\ref{mainR} can be proved using  the formula in \cite[Proposition~4.33]{GFS} (we recovered it in \eqref{gein7}) applying the soliton conditions and the strong maximum principle following more or less the steps in our proof. To the best of our knowledge, the formulation and a complete proof of Theorem~\ref{mainR} appears for the first time in this paper.
\end{rmrk}


\end{document}